\documentclass[final,1p,times]{elsarticle}

\journal{Journal of Combinatorial Theory, Series A}

\makeatletter
\def\ps@pprintTitle{%
    \let\@oddhead\@empty
    \let\@evenhead\@empty
    \def\@oddfoot{}%
    \let\@evenfoot\@oddfoot}
\makeatother

\usepackage[utf8]{inputenc}
\usepackage{amsmath}
\usepackage{amssymb}
\usepackage{amsthm}
\usepackage{bbm}
\usepackage{todonotes}
\usepackage{hyperref}
\usepackage{xparse}

\newcommand{\corr}{\mathrm{CORR}}
\DeclareMathOperator{\conv}{conv}
\DeclareDocumentCommand\setdef{mo}{\left\{#1\IfNoValueTF{#2}{}{ : #2}\right\}}
\newcommand{\R}{\mathbb{R}}
\newcommand{\scalprod}[2]{\langle{#1},{#2}\rangle}

\newcommand{\udisj}{\mathrm{UDISJ}}
\newcommand{\disjpairs}{\mathcal{D}}

\newtheorem{theorem}{Theorem}
\newtheorem{lemma}{Lemma}

\begin{document}

%
%
\begin{frontmatter}

    \title{A Short Proof that the Extension Complexity of the Correlation Polytope Grows Exponentially}

    \author{Volker Kaibel}
    \ead{kaibel@ovgu.de}
    \author{Stefan Weltge\fnref{weltge}}
    \ead{weltge@ovgu.de}
    \fntext[weltge]{Partially funded by the German Research Foundation (DFG): ``Extended Formulations in Combinatorial
    Optimization'' (KA 1616/4-1)}

    \address{Otto-von-Guericke-Universit\"at~Magdeburg, Germany}


    \begin{abstract}
        We establish that the extension complexity of the $n\times n$ correlation polytope is at least $1.5\,^n$ by a
        short proof that is self-contained except for using the fact that every face of a polyhedron is the intersection
        of all facets it is contained in. The main innovative aspect of the proof is a simple combinatorial argument
        showing that the rectangle covering number of the unique-disjointness matrix is at least $1.5^n$, and thus the
        nondeterministic communication complexity of the unique-disjointness predicate is at least $.58n$. We thereby
        slightly improve on the previously best known lower bounds $1.24^n$ and $.31n$, respectively.
    \end{abstract}


\end{frontmatter}

%
%

\section{Introduction}

\noindent
The concept of extended formulations aims at writing polytopes as affine images of polyhedra of lower complexity. In
particular, for a polytope $ P $, one is interested in its \emph{extension complexity}, i.e., the smallest number of
facets of any polyhedron whose affine image is $ P $. As the first explicit example of a $0/1$-polytope whose extension
complexity is not bounded by a polynomial in its dimension, Fiorini~et~al.~\cite{FioriniMPTW12} showed that the
extension complexity of the \emph{correlation polytope}
\[
    \corr(n) := \conv \setdef{y \in \{0,1\}^{n \times n}}[y_{ij} = x_i x_j \ \forall \, i,j \in [n], \,x \in \{0,1\}^n]
\]
grows exponentially in $ n $. Since $ \corr(n) $ can be found as an affine image of a face of many other combinatorial
polytopes of similar dimension, this result has been used to show that the extension complexities of polytopes such as
traveling salesman polytopes~\cite{FioriniMPTW12}, certain stable set polytopes~\cite{FioriniMPTW12}, certain knapsack
polytopes~\cite{AvisT13,PokuttaV13}, and other polytopes associated with NP-hard optimization problems~\cite{AvisT13}
are also not bounded polynomially. Independently of the correlation polytope, Rothvoß~\cite{Rothvoss13} recently even
established an exponential lower bound on the extension complexity of the perfect matching polytope.

The proof of the statement on $ \corr(n) $ given in~\cite{FioriniMPTW12} follows a strategy developed
in~\cite{Yannakakis91} and uses a lower bound on the rectangle covering number of the unique-disjointness matrix
obtained in~\cite{Wolf03}, which essentially is due to~\cite{Razborov92}.  This amounts to a rather involved proof in
total, leaving it unclear how ``deep'' the result actually is (while its great relevance is out of discussion, of
course).

The aim of this paper is to provide a short combinatorial, self-contained (except for using the fact that every face of
a polyhedron is the intersection of all facets containing it) proof showing that the extension complexity of $ \corr(n)
$ is at least $ 1.5\,^n $. The main new contribution of the proof is a simple combinatorial argument (see the
half-a-page proof of Thm. \ref{thm:main}) instead of using~\cite{Wolf03,Razborov92}. Furthermore, the lower bound $
1.5\,^n $ improves slightly upon the previously best known one $ 1.24\,^n $ following from~\cite{BraunP13}.

%
%

\section{The Main Proof}
\label{sec:proof}

\noindent
For a nonnegative integer $ n $ we set $ [n] := \setdef{1,\dotsc,n} $ and define $ 2^{[n]} $ as the set of all subsets
of $ [n] $. The Euclidian scalar product of two vectors $ v,w $ is denoted by $ \scalprod{v}{w} = \sum_i v_i w_i $.
Further, for a set $ a \subseteq [n] $ let $ \chi(a) \in \{0,1\}^n $ be its characteristic vector, i.e., $ \chi(a)_i = 1
$ if and only if $ i \in a $. For a set $ b \subseteq [n] $ let $ y^b \in \{0,1\}^{n \times n} $ be the
0/1-matrix with $ y^b_{ij} = 1 $ if and only if $ i \in b $ and $ j \in b $ hold. With this notation, we have that $
\corr(n) = \conv \{ y^b : b \subseteq [n] \} $.

We first extract the single combinatorial property of $ \corr(n) $ that is relevant for the proof and then, by a few
polyhedral arguments, establish a general lower bound on the extension complexity of $ \corr(n) $ in terms of sizes of
so-called coverings. This part is basically a compact reformulation of known arguments.

\begin{lemma}
    For every $ a \subseteq [n] $ there is a face $ F_a $ of $ \corr(n) $ such that
    \[
        y^b \in F_a \iff |a \cap b| = 1
    \]
    holds for all $ b \subseteq [n] $.
\end{lemma}
\begin{proof}
    For a set $ a \subseteq [n] $, let $ \pi_a(x) \in \R[x_i:i\in [n]] $ be the quadratic polynomial $
    (\scalprod{\chi(a)}{x} - 1)^2 $ with variable vector $ x=(x_1,\dotsc,x_n) $. Denote by $ \Pi_a(y) \in \R[y_{ij} :
    i,j \in [n]] $ the linear polynomial arising from $ \pi_a(x) $ by substituting each monomial $ x_i x_j $ by $ y_{ij}
    $ and each monomial $ x_i $ by $ y_{ii} $. Due to $ y^b_{ij} = \chi(b)_i \cdotp \chi(b)_j $ and $ y^b_{ii} = \chi(b)_i $ we
    have $\Pi_a(y^b) = \pi_a(\chi(b)) \ge 0 $ for each $ b \subseteq [n] $. This implies that the linear inequality $
    \Pi_a(y) \ge 0 $ is valid for $ \corr(n) $ and hence defines a face $ F_a $ of $ \corr(n) $. Note that a point $ y^b
    $ is contained in $ F_a $ if and only if $ \scalprod{\chi(a)}{\chi(b)} = 1 $, i.e., $ | a \cap b | = 1 $ holds.
\end{proof}

\noindent
Let us define the set $ \disjpairs(n) := \{ (a,b) \in 2^{[n]} \times 2^{[n]} : a \cap b = \emptyset \} $ of pairs of
disjoint subsets of $ [n] $. A set $ R \subseteq \disjpairs(n) $ is called  \emph{valid} if it satisfies
\begin{equation}
    \label{eq:validity}
    \forall \ (a,b), \, (a',b') \in R: \quad | a \cap b' | \neq 1.
\end{equation}
Further, we say that a set $ R_1,\dotsc,R_k $ of valid sets in $ \disjpairs(n) $ is a \emph{covering} of $
\disjpairs(n) $ of size $ k $ if
\[
    \disjpairs(n) \subseteq \bigcup_{i=1}^k R_i
\]
holds.

\begin{lemma}
    \label{lem:induced-covering}
    Let $ Q $ be a polyhedron having $ f $ facets such that $ \corr(n) $ is an affine image of $ Q $. Then there exists
    a covering of $ \disjpairs(n) $ of size $ f $.
\end{lemma}
\begin{proof}
    Let $ p $ be an affine map such that $ p(Q) = \corr(n) $. For every facet $ G $ of $ Q $ let us define the set
    \[
        R_G := \setdef{(a,b) \in \disjpairs(n)}[
            p^{-1}(F_a) \cap Q \subseteq G, \, p^{-1}(y^b) \cap Q \not \subseteq G].
    \]
    First, note that $ R_G $ is valid because otherwise there exist $ (a,b), \, (a',b') \in R_G $ with $ | a \cap b' |
    = 1 $, which implies $ y^{b'} \in F_a $, and hence we obtain $ p^{-1}(y^{b'}) \cap Q \subseteq
    p^{-1}(F_a) \cap Q \subseteq G, $ a contradiction to the definition of $ R_G $.

    Second, we claim that $ \setdef{R_G}[G \text{ facet of } Q] $ is a covering of $ \disjpairs(n) $. Towards this end,
    let $ (a,b) \in \disjpairs(n) $. Observe that $ p^{-1}(F_a) \cap Q $ is a face of $ Q $ and let $ G_1,\dotsc,G_k $
    be the facets of $ Q $ containing $ p^{-1}(F_a) \cap Q $. As $ y^b \notin F_a $ and $ p^{-1}(F_a) \cap Q =
    \bigcap_{i=1}^k G_i $, we obtain that there exists some $ i \in [k] $ for which $ p^{-1}(y^b) \cap Q \subseteq G_i $
    does not hold. Thus, we obtain $ (a,b) \in R_{G_i} $.
\end{proof}

\noindent
We are now ready to prove our main result:

\begin{theorem}
    \label{thm:main}
    The extension complexity of $ \corr(n) $ is at least $ 1.5\,^n $.
\end{theorem}
\begin{proof}
    By Lemma \ref{lem:induced-covering}, it suffices to show that any covering of $ \disjpairs(n) $ has size at least
    $ 1.5\,^n $. Therefore, let $ \varrho(n) $ be the largest cardinality of any valid subset of $ \disjpairs(n) $. By
    the fact that any covering of $ \disjpairs(n) $ must have size of at least $ \frac{|\disjpairs(n)|}{\varrho(n)} $
    and the fact that $ |\disjpairs(n)| = 3^n $, it remains to show that $ \varrho(n) \leq 2^n $, which we will
    establish by showing that $ \varrho(n) \leq 2 \varrho(n-1) $ holds for all $ n \geq 1 $. (Note that $ \varrho(0) = 1
    $ since the only valid subset of $ \disjpairs(n) $ is $ \{ ( \emptyset, \emptyset ) \} $.)

    Towards this end, let $ R \subseteq \disjpairs(n) $ be valid (with $ n \geq 1 $) and let us
    define the following two sets:
    \begin{align*}
        R_1 & := \left( \setdef{(a,b) \in R}[n \in a] \cup \setdef{(a,b) \in R}[(a \cup \{n\},b) \notin R] \right) \,
        \cap \, ([n] \times [n-1]) \\
        R_2 & := \left( \setdef{(a,b) \in R}[n \in b] \cup \setdef{(a,b) \in R}[(a,b \cup \{n\}) \notin R] \right) \,
        \cap \, ([n-1] \times [n])
    \end{align*}
    Further, let us define the function $ f \colon R \to \disjpairs(n-1) $ with $ f(a,b) := (a \setminus \{n\}, b
    \setminus \{n\}) $. Since $ R_1 \subseteq R $ is valid and since $ R_1 \subseteq [n] \times [n-1] $, $ f(R_1) $ is
    valid. Similarly, $ f(R_2) $ is also valid. Further, by the definition of $ R_i $, $ f $ is injective on $ R_i $ for
    $ i = 1,2 $. By induction, we hence obtained that
    \[
        |R_1| + |R_2| = |f(R_1)| + |f(R_2)| \leq 2 \varrho(n-1) = 2^n.
    \]
    Thus, it suffices to show that each $ (a,b) \in R $ is contained in $ R_1 \cup R_2 $. Since $ a \cap b = \emptyset $, we
    have that $ (a,b) \subseteq ([n] \times [n-1]) \cup ([n-1] \times [n]) $. Thus, if $ n \in a \cup b $, we clearly
    have that $ (a,b) \in R_1 \cup R_2 $. It remains to show that for any $ (a,b) \in R $ with $ n \notin a \cup b $, we
    cannot have that $ (a \cup \{n\},b) \in R $ and $ (a,b \cup \{n\}) \in R $. Indeed, this is true since, otherwise,
    the validity of $ R $ would imply
    \[
        1 \neq |(a \cup \{n\}) \cap (b \cup \{n\})| = | \{n\} | = 1,
    \]
    a contradiction.
\end{proof}

%
%

\section{Remarks on Related Results}

\subsubsection*{From the Perspective of Communication Complexity}

\noindent
Using the terminology from the theory of communication complexity, the proof of Theorem \ref{thm:main} shows that the
\emph{rectangle covering number} of the \emph{unique-disjointness matrix} $ \udisj(n) $ (see, e.g., \cite{Jukna12})
is at least $ 1.5\,^n $. To see that, observe that our notion of valid sets corresponds to sets of 1-entries in $
\udisj(n) $ that can be covered simultaneously by one rectangle. In particular, this implies that the
\emph{nondeterministic communication complexity} of the \emph{unique-disjointness} predicate is at least $
\log_2(1.5\,^n) \geq .58n $. For the background of these remarks, we refer to \cite{KushilevitzN06} or \cite{Jukna12}.

\subsubsection*{Applicability to the Matching Polytope}

\noindent
Most superpolynomial lower bounds on the extension complexities of combinatorial polytopes are a direct consequence of
the fact that the extension complexity of the correlation polytope grows exponentially and hence can also be derived
from our argumentation. In contrast to this, Rothvoß'~\cite{Rothvoss13} result on an exponential lower bound on the
extension complexity of the (perfect) matching polytope of the complete graph seems to be of a considerably more
complicated nature. It follows already from \cite{Yannakakis91} that this result cannot be deduced from the results on
the correlation polytope in a similar manner as it is possible for, say, the TSP polytope. In fact, Rothvoß' approach
exploits more than the mere combinatorial structure of the matching polytopes. The ideas underlying the proof presented
in this paper seem to be of little use in that context, leaving wide open the question for a similarly simple proof of
the fact that the extension complexity of the perfect matching polytope cannot be bounded polynomially.

%
%
\paragraph*{Acknowledgements}
We thank Yuri Faenza and Kanstantsin Pashkovich for their helpful remarks on an earlier version of this paper.

\bibliographystyle{elsarticle-num}
\bibliography{references}

\begin{thebibliography}{10}
\expandafter\ifx\csname url\endcsname\relax
  \def\url#1{\texttt{#1}}\fi
\expandafter\ifx\csname urlprefix\endcsname\relax\def\urlprefix{URL }\fi
\expandafter\ifx\csname href\endcsname\relax
  \def\href#1#2{#2} \def\path#1{#1}\fi

\bibitem{FioriniMPTW12}
S.~Fiorini, S.~Massar, S.~Pokutta, H.~R. Tiwary, R.~de~Wolf, Linear vs.
  semidefinite extended formulations: exponential separation and strong lower
  bounds, in: STOC, 2012, pp. 95--106.

\bibitem{AvisT13}
D.~Avis, H.~R. Tiwary, {On the Extension Complexity of Combinatorial
  Polytopes}, in: F.~V. Fomin, R.~Freivalds, M.~Z. Kwiatkowska, D.~Peleg
  (Eds.), Automata, Languages, and Programming, Vol. 7965 of Lecture Notes in
  Computer Science, Springer Berlin Heidelberg, 2013, pp. 57--68.

\bibitem{PokuttaV13}
S.~Pokutta, M.~V. Vyve, A note on the extension complexity of the knapsack
  polytope, Oper. Res. Lett. 41~(4) (2013) 347--350.

\bibitem{Rothvoss13}
T.~Rothvoß, \href{http://arxiv.org/abs/1311.2369}{The matching polytope has
  exponential extension complexity}, arXiv:1311.2369 (2013).
\newline\urlprefix\url{http://arxiv.org/abs/1311.2369}

\bibitem{Yannakakis91}
M.~Yannakakis, {Expressing Combinatorial Optimization Problems by Linear
  Programs}, J. Comput. Syst. Sci. 43~(3) (1991) 441--466.

\bibitem{Wolf03}
R.~de~Wolf, Nondeterministic quantum query and communication complexities, SIAM
  J. Comput. 32~(3) (2003) 681--699.

\bibitem{Razborov92}
A.~A. Razborov, On the distributional complexity of disjointness, Theoretical
  Computer Science 106~(2) (1992) 385--390.

\bibitem{BraunP13}
G.~Braun, S.~Pokutta, Common information and unique disjointness, in:
  Proceedings of the 54th Symposium on Foundations of Computer Science (FOCS),
  2013, pp. 688--–697.

\bibitem{Jukna12}
S.~Jukna, {Boolean Function Complexity: Advances and Frontiers}, Springer,
  2012.

\bibitem{KushilevitzN06}
E.~Kushilevitz, N.~Nisan, {Communication Complexity}, Cambridge University
  Press, 2006.

\end{thebibliography}

\end{document}